\documentclass[11pt]{amsart}
\usepackage{pinlabel,amsmath}

\usepackage{url}
\usepackage{bm}
\usepackage[left=1in,top=1in,right=1in,bottom=.8in,head=.2in]{geometry}

\usepackage{enumitem}
\usepackage[textsize=scriptsize]{todonotes}

\usepackage{fancyhdr}
\newtheorem{thm}{Theorem} 

\newtheorem{theorem}{Theorem}[section] 
\newtheorem*{theorem*}{Theorem}
\newtheorem{lemma}[theorem]{Lemma}
\newtheorem*{lemma*}{Lemma}

\newtheorem*{corollary*}{Corollary}
\theoremstyle{definition}
\newtheorem{definition}[theorem]{Definition}

\newtheorem*{remark*}{Remark}
\newtheorem*{remarks*}{Remarks}
\newtheorem*{addenda*}{Addenda}

\newcommand{\Z}{\mathbb Z}

\begin{document}

\title{Counting Regions in Billiard Trajectories}
\author[Dave Auckly]{Dave Auckly$^{1}$}
\address{Department of Mathematics\newline\indent Kansas State University\newline\indent  Manhattan,
Kansas 66506}
\email{dav@math.ksu.edu}
\author[Betsy Giles]{Betsy Giles$^{2}$}
\address{Department of Mathematics\newline\indent Kansas State University\newline\indent  Manhattan,
Kansas 66506}
\email{betsylynn@ksu.edu}
\thanks{$^1$Partially supported by Simons Foundation grant 585139 and NSF grant DMS-1952755.
$^2$Partially supported by an Arts and Sciences Research Scholarship and the K-State I-Center. }
\maketitle

\parskip 6pt

\vskip-.4in
\vskip-.4in

\section{Introduction}

Mathematical billiards is an ideal model of  a billiard game. 
One can trace the path of a ``billiard ball" as it travels on a billiard table. This path records where the ball travels once it is set in motion. In the mathematical model of billiards, the billiard ball is infinitesimally small, frictionless, and all collisions are elastic. When the ball reaches the table edges, it reflects with an angle equal to the angle of incidence. Mathematical billiards can be studied on  tables of any shape. 

In this paper we study paths that start and end in a corner (not necessarily the same corner). Each such path or trajectory produces a geometric pattern of parallelograms. By suitably scaling the length of the rectangle, one may assume that the path starts at an  angle of $45$ degrees to an edge.  This convention defines a natural aspect ratio for the sides of the rectangle for such a path. Continue the motion of the ball as it reflects off each edge at a $45$ degree angle and mark the trajectory. This continues until the ball reaches a corner, ending the ball trajectory. The resulting geometric pattern contains slanted squares and rectangles. The smallest such squares are called \emph{atomic squares}. Larger slanted rectangles composed of atomic squares are called \emph{molecular rectangles}. See Figure~\ref{Def} where the atomic square is green and the molecular rectangle is blue. In Figure~\ref{Def}, there are exactly three atomic squares and five molecular rectangles (three $1 \times 1$, one $2\times 1$ and one $1 \times 2$). In this paper we count the number of atomic squares and molecular rectangles formed by billiard paths as the dimensions of the billiard table changes. 
\begin{figure}[htb]
\centering
\includegraphics[width=2.2in]{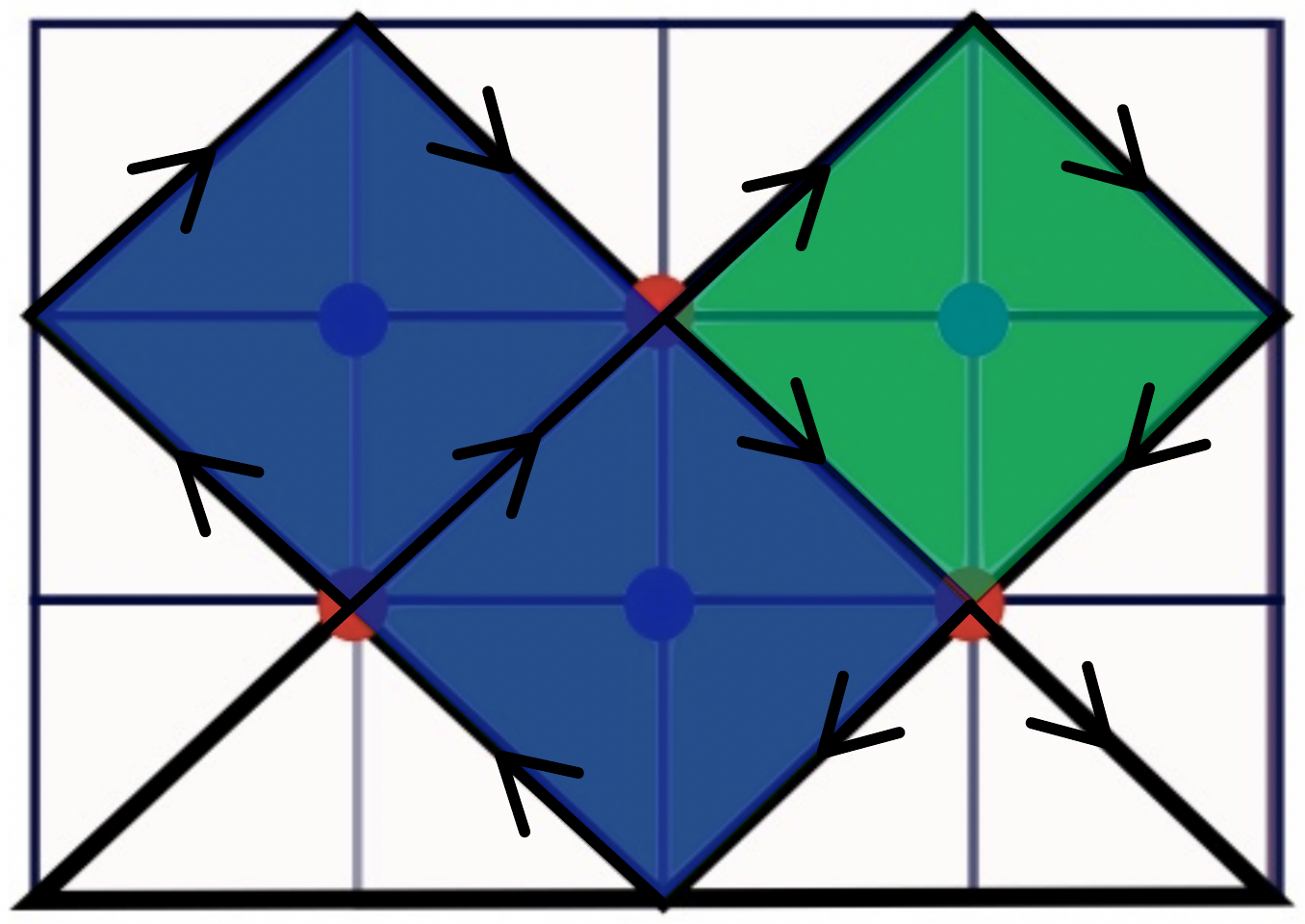}
\caption{Example of atomic squares and molecular rectangles}
\label{Def}
\end{figure}
This work was motivated by the ``On Reflection" activity from Tanton's book \cite{Tanton}. This activity has been used as a recurring favorite in the Navajo Nation Math Circles project. A teacher asked for the number of slanted rectangles in such billiard trajectories during a teacher workshop. The teachers made some progress on the question, and the question was further studied in the undergraduate research project leading to this paper. A description of the Navajo Nation Math Circles and a number of mathematical activities may be found in \cite{AKSS}. Our main results provide a count of the number of atomic squares and molecular rectangles in any rectangular billiard trajectory.

\begin{thm}\label{thmA}
The number of atomic squares in an $a \times b$ rectangle is

\[
\frac{(a-\text{\rm gcd}(a,b))(b-\text{\rm gcd}(a,b))}{2\times\text{\rm gcd}(a,b)^2}
\]

\end{thm}

\begin{thm}\label{thmB}

The number of molecular rectangles in an $a \times b$ rectangle, with $a > b$  is
\[
\frac{1}{24}q(q^2-1)(2p-q) \] 
\vspace{-.2in}
\[
\text{\rm where } p= \frac{a}{\text{\rm gcd}(a,b)} 
\quad \text{\rm and} \quad q= \frac{b}{\text{\rm gcd}(a,b)}.
\]

\end{thm}

As an example refer to Figure~\ref{Checker} which displays the billiard path in a $p\times q = 5\time 4$ grid. The formula for atomic squares gives $(5-1)(4-1)/2 = 6$ and the formula for molecular rectangles gives $4(4^2-1)(2\cdot 5 - 4)/24 = 15$. These agree with what one obtains by counting the squares and rectangles in the figure.

\section*{Initial reductions and atomic squares}

We begin by reducing the general counting problem to a standard form. Our first observation is that we may assume that the trajectory is at a 45 degree angle to the edge of the region. We place the rectangle with its edges parallel to the coordinate axes of a coordinate system with the lower left corner of the rectangle at the origin. By symmetry we may assume that the trajectory starts at the origin. We call a trajectory trivial if it just follows an edge from the origin to a corner.

\begin{lemma}
Any non-trivial billiard trajectory starting at a corner in a rectangle produces a pattern of parallelograms combinatorially equivalent to the pattern produced by a trajectory making a 45 degree angle with the edge.
\end{lemma}

\begin{proof}
Let $m$ denote the slope of the original trajectory. Since the trajectory is non-trivial and the rectangle is in the first quadrant we see that $0 < m < \infty$. Now use the transformation $(x,y)\mapsto (x,y/m)$. This is just a vertical change-of-scale which maps vertical lines to vertical lines, horizontal lines to horizontal lines, and parallelograms, to parallelograms. This changes the trajectory of slope $m$ to a trajectory of $1$, and a trajectory of slope $-m$ to a trajectory of slope $-1$.

\begin{figure}[htb]
\centering
\includegraphics[width=2.2in]{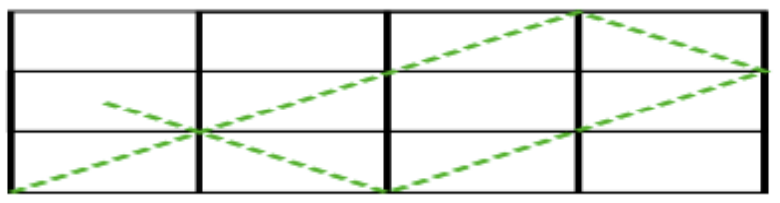} \hskip.3in
\includegraphics[width=2.2in]{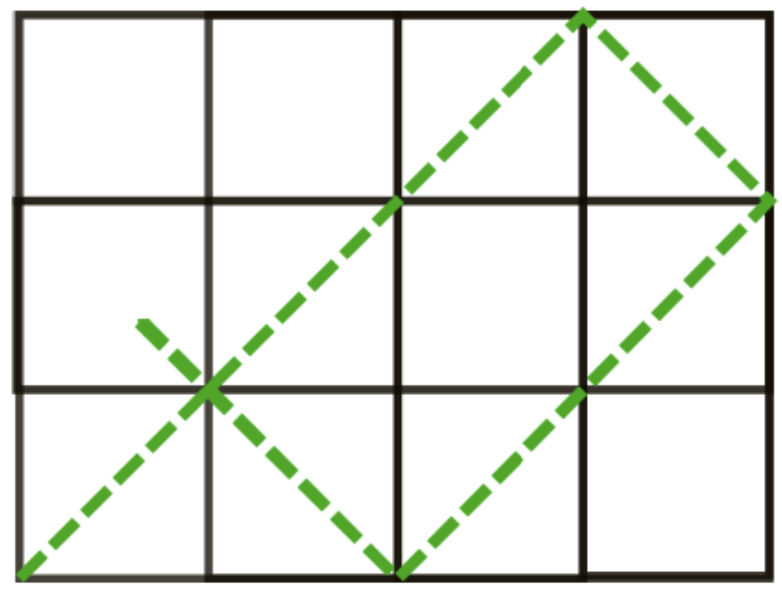}
\caption{Rescaling a billiard table.}
\label{Ratio}\label{ratio}
\end{figure}

\end{proof}

Denote the edge lengths of the normalized rectangle $a$ and $b$. A non-trivial rectangle will have $a > 0$ and $b > 0$. Since the normalized trajectory has a slope of $1$, the $x$ and $y$ coordinates of the first point where it hits the boundary must be equal. In particular, they form a rational ratio. As later slopes are always $\pm 1$, all other points where the trajectory meets the side must have a rational ratio. 
We now wish to reduce to the case of a $p\times q$ rectangle with $\text{gcd}(p,q) = 1$.  
The path of the billiard ball simply depends on the ratio of the edge lengths to each other. Indeed, while grid lines help track the pattern, their presence does not determine the path of the billiard ball.

Given an $a\times b$ rectangular table with rational side length ratio, we can find $p$ and $q$ with the $\text{gcd}(p,q)=1$ such that $a=pd$ and $b=qd$. Without changing any part of the billiard ball path, we can subdivide the edges into segments of length $d$. This subdivision does not change the path, it just adds auxillary lines to make it clear that the combinatorial pattern only depends upon the ratio of the side lengths. See Figure~\ref{ratio}.

\begin{figure}[htb]
\centering
\includegraphics[width=2.2in]{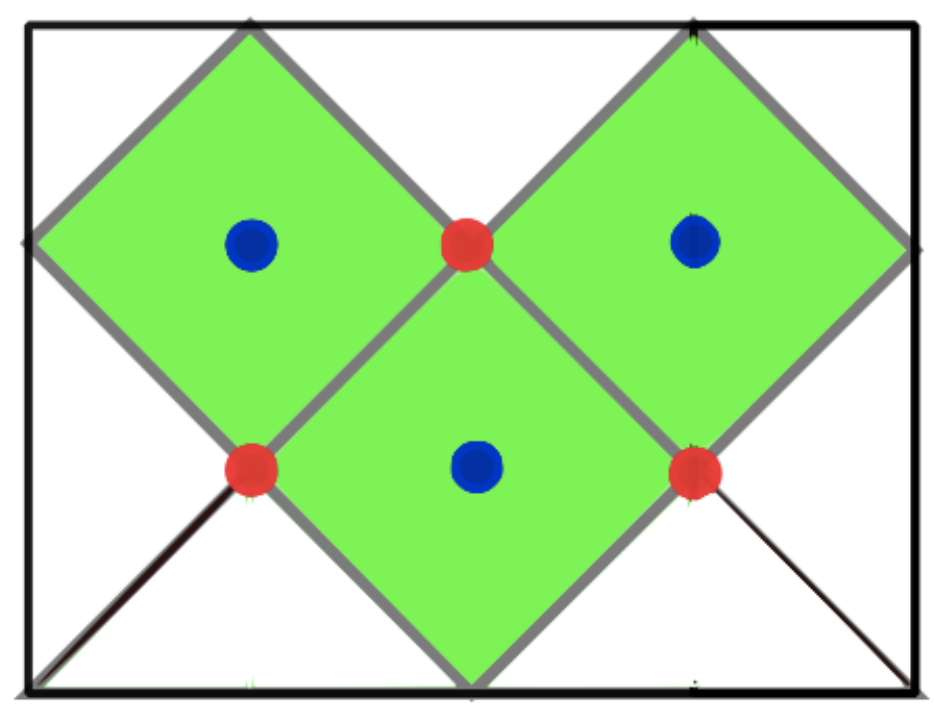} \hskip.3in
\includegraphics[width=2.2in]{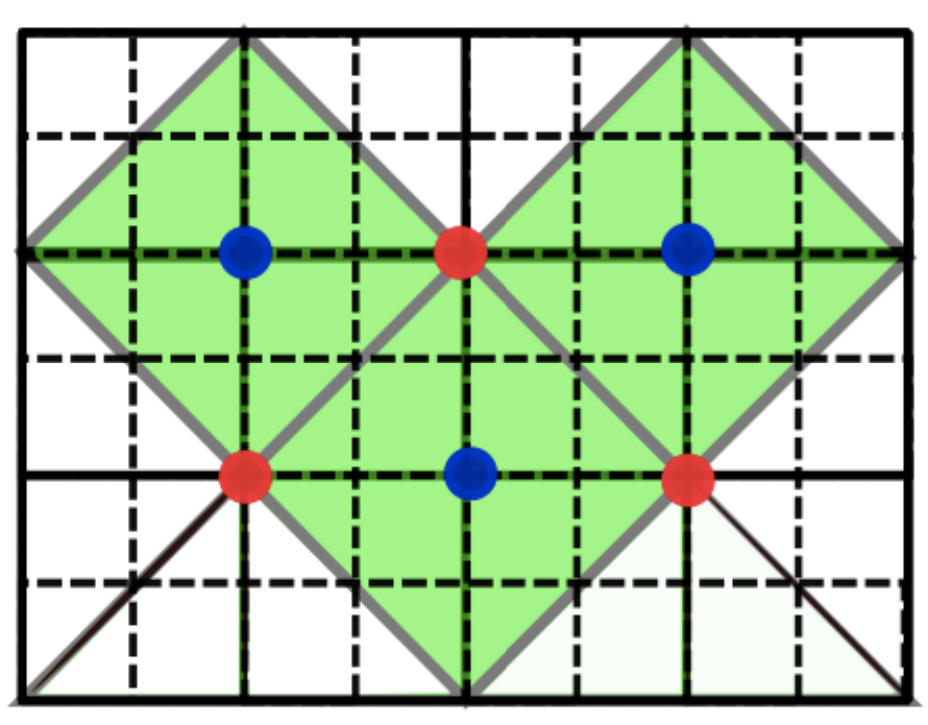}
\caption{Example of a 4:3 side ratio resulting in three atomic squares.}
\label{Ratio}\label{ratio}
\end{figure}

By looking at the path of the billiard ball, one notices an apparent relationship between the grid points and the  atomic squares. It is clear that no atomic square can be centered on any of the grid points along the exterior of the grid. Thus only the interior grid points are of interest. If the length of the pool table is $p$ elementary segments, each segment has a unique right end point that belongs to that segment. The right end point of the final segment is an exterior edge. Since the exterior points are being excluded one concludes that there are $(q-1)$ interior right endpoints along the length. See Figure~\ref{Subtract}. Similarly, there are $(p-1)$ interior grid points along the width. It follows that a $p\times q$ rectangle has $(p-1)(q-1)$ interior grid points.  With this background we can now prove  Theorem~\ref{thmA}.

\begin{figure}[htb]
\centering
\includegraphics[width=1.3in]{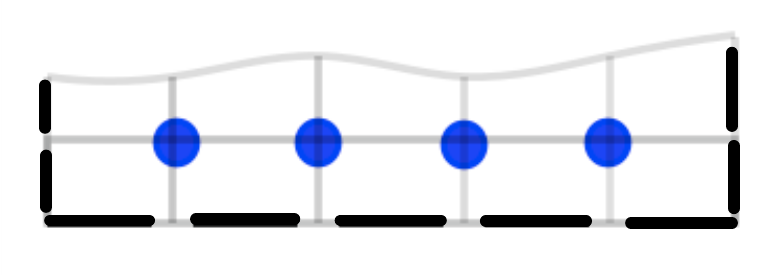}
\caption{An edge of $5$ segments with $4$ interior grid points.}
\label{Subtract}
\end{figure}

\noindent{\it Proof of Theorem A -- Outline.}
The interior grid points that the path doesn't pass through are exactly the centers of the atomic squares. There is an intuitive explanation this and for  why the billiard ball path passes through exactly half of the grid points and doesn't pass through the other half. We provide this explanation here and then provide a more detailed analysis of this fact in the next section. When looking at the grid, once the exterior points have been eliminated from examination, an interior grid remains. Looking at the interior grid, each grid point can be colored in a checkerboard pattern, with one color corresponding to grid points that the line passes through and the other color corresponding to the grid points that form the centers of the atomic squares (Figure~\ref{Checker}). Since  $\text{gcd}(p,q)=1$ at least one of $p$ or $q$ must be odd. Thus one of $(p-1)$ or $(q-1)$ must be even. The checkerboard coloring will alternate in that direction ending evenly, so exactly one-half of the interior grid points along lines in that direction will be one color, and the other half will be the other color. In Figure~\ref{Checker} there are four interior grid points on each line parallel to the side of length $5$. With alternate colors exactly two out of four in each row correspond to atomic squares. It follows that the number of atomic squares, which we will see is the same as the number of interior grid points that the path misses is exactly one half of the number of interior grid points, so
\[
\text{Number of Atomic Squares} = \frac{(p-1)(q-1)}{2}.
\]
To obtain the formula in the statement of the theorem, substitute $p = a/\text{gcd}(a,b)$, $q = q/\text{gcd}(a,b)$ and multiply the numerator and denominator of the large fraction by $\text{gcd}(a,b)^2$.

\begin{figure}[htb]
\centering
\includegraphics[width=2.5in]{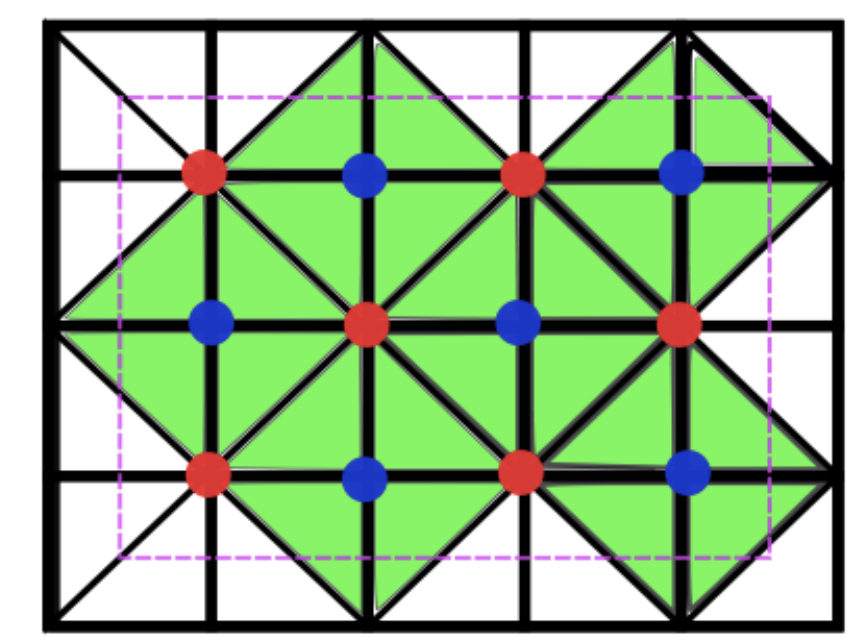}
\caption{Example of 5 x 4 table showing checkerboard pattern with six atomic squares.}
\label{Checker}
\end{figure}

\section*{Grid Points on the Trajectory -- Extra details for the Proof of Theorem A}

To understand why the billiard ball path passes through exactly half of the grid points, we will first analyze something we call asteroid paths. The name for these paths come from a video game that was popular in the early 1980s. In this game, asteroids moved in straight lines in a rectangular region. When an asteroid reached an edge of the rectangle it instantly transported to the corresponding point on the opposite edge and continued traveling in the same direction. See Figure~\ref{astf}.

\begin{figure}[htb]
\centering
\includegraphics[width=2.5in]{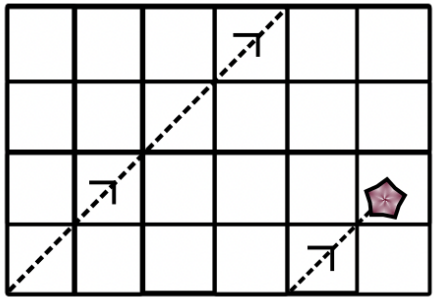}
\caption{An asteroid path.}
\label{astf}
\end{figure}

Just as one can normalize non-trivial billiard paths to have slope $\pm 1$, one can normalize non-trivial asteroid paths to have slope $1$. (In this case the direction of the path is always constant, so we do not need to consider positive and negative slopes. It is this feature that makes asteroid paths easier to analyze. At any rate we only need to study such paths, so from here forward we assume that all asteroid paths have slope $1$.  Asteroid paths are easier to describe mathematically, but still have the property that they pass through exactly one half of the grid points in suitable rectangular regions. The basic observation that allows one to describe asteroid paths is that when such a path is at an interior grid point with coordinates $(x,y)$ and the path has a slope of one, the next point the path will reach is $(x+1,y+1)$. To model the transportation from one edge to the opposite edge, we use cyclic groups. Let  ${\mathbb{Z}_d}$ represent the normalized remainders when one divides by $d$. Thus these are the numbers $\{0,1,\cdots,d-1\}$. When we add two such numbers we can always put the answer in this range by removing extra copies of $d$, i.e., divide by $d$ and take the remainder. For example in $\Z_6$ one has $4+3 = 7 = 7 - 6 = 1$. Here we say $4+3$ is equal to $1$ modulo $6$ and more properly write $4+3 \equiv 1 \quad( \text{mod}\  6)$. We can now give a formal definition of an asteroid path.

\begin{definition}\label{astd}
The \emph{asteroid path} on a $2p \times 2q$ grid with the $\text{gcd}(p, q)=1$ is the path, \hfill\newline
$\alpha(n):  \mathbb{Z}_{2pq} \rightarrow {\mathbb{Z}_{2p}}\times {\mathbb{Z}_{2q}}$,  described by the following function.
\[
\alpha(n) = (n,n) \ (\text{mod}\ (2p,2q)).
\]
\end{definition}
 
This model of the  $2p \times 2q$ asteroid rectangle is an example of a group, and the asteroid path is an example of a group homomorphism. It is this additional structure that makes the asteroid paths easier to analyze.
Label each grid point with its coordinates. It is apparent that the path only goes through points that have an ordered pair that sum to an even number and thus are equivalent to zero modulo 2. 
This is because the path mapped always has a slope of one, so as one coordinate increases by one,  the other will as well. Subtracting $2p$ or $2q$ upon meeting an edge will not change the parity. This is the reason we need to assume that both sides of the rectangle are an even number of segments. If this is not the case and one side has odd length, the path can pass through points with odd coordinate sum. (See Figure~\ref{non-ast}.) Just because the path only passes through points with even coordinate sum, does not mean the path will pass through every such point. Indeed, for a $4\times 4$ rectangle, the path misses several of the points with even coordinate sum. (See Figure~\ref{non-ast}.) Notice that by our definition the point $(x,2q)$ is the same as the point 
$(x,0)$. We mark both points in our figure when a path runs through either. When counting points, it makes sense to just count the points with coordinates $(x,y)$ satisfying $0 \le x < 2p$ and $0 \le y < 2q$.

\begin{figure}[htb]
\centering
\includegraphics[width=2.5in]{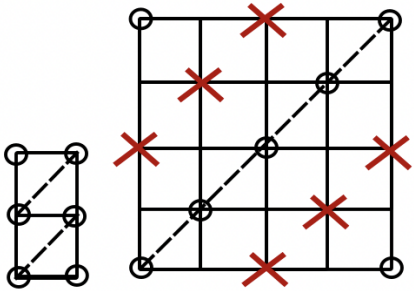}
\caption{The $1\times 2$ and $4\times 4$ asteroid paths.}
\label{non-ast}
\end{figure}

We now show that asteroid paths in suitable rectangles pass through exactly half of the grid points.
\begin{lemma} \label{evaft} An asteroid path in a  $2p \times 2q$ rectangle with $\text{gcd}(p,q) = 1$ passes through exactly one half of the grid points. In particular, exactly the grid points $(x,y)$ with $x+y \equiv 0 \ (\text{mod}\  2)$.
\end{lemma}

\begin{proof}
Let $\Gamma := \{(x,y)\in {\mathbb{Z}_{2p}}\times {\mathbb{Z}_{2q}}\,|\, x+y \equiv 0 \ (\text{mod}\  2)\}$. We have seen that the map $\alpha$ takes values in $\Gamma$. Notice that the number of elements in $\Gamma$ is $\frac12(2p)(2q) = 2pq$. This is the same as the number of elements in $\Z_{2pq}$. Thus, if we show that $\alpha$ is injective, we will be able to conclude that it is surjective. Now $\alpha(n) = (0,0)$ implies that $n \equiv 0 \ (\text{mod}\  2p)$ and $n \equiv 0 \ (\text{mod}\  2q)$. Thus $n$ contains a factor of $2$ and all of the prime divisors of $p$ counted with multiplicity and all of the prime divisors of $q$ counted with multiplicity. (This is where we need to assume that $\text{gcd}(p,q) = 1$). If follows that
$n \equiv 0 \ (\text{mod}\  2pq)$. This will imply that $\alpha$ is injective. Indeed, if $\alpha(n)=\alpha(m)$, then $\alpha(n-m)=0$, so $n-m \equiv 0 \ (\text{mod}\  2pq)$  and $n\equiv m  \ (\text{mod}\  2pq)$ establishing that $\alpha$ is injective and we are done.
\end{proof}

We now show how the analysis of asteroid paths can be applied to billiard trajectories. The basic idea is to fold the $2p\times 2q$ asteroid rectangle in half in the vertical direction and then fold it in half in the horizontal direction to obtain the billiard path on the $p\times q$ rectangle. See Figure~\ref{fold}.

\begin{figure}[htb]
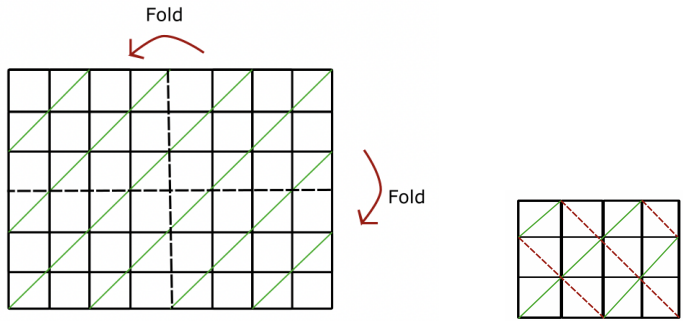

\centering
\includegraphics[width=2.34in]{Fold1.pdf} \hskip.3in
\includegraphics[width=.99in]{Folded.pdf}
\caption{Folding a grid to convert astroids into billiards.}
\label{fold}
\end{figure}

To describe this folded model for the pool table, we examine the map from the asteroid rectangle to the pool table. 
This map must send all points to points $(x',y')$ such that $0 \leq x' \leq p$ and $0 \leq y' \leq q$. Any point that satisfies this condition is mapped to itself. For the other points we use the following cases:

Case 1: when the $y$-coordinate is outside of the range, but the $x$-coordinate is in the range:
\[
(x, y) \rightarrow (x, 2q-y)
\]

Case 2: when the $x$-coordinate is outside of the range, but the $y$-coordinate is in the range:
\[
(x,y) \rightarrow (2p-x, y)
\]

Case 3: when both coordinates are out of the  range:
\[
(x,y) \rightarrow (2p-x, 2q-y)
\]

An alternate, and more efficient description is to represent the cyclic groups by $\Z_d/\sim$ where we define two integers to be equivalent if they differ by a multiple of $d$: $x\sim x'$ exactly when $x-x'\in d\Z$. We now define an equivalence relation on $\Z_{2p}\times\Z_{2q}$ by
\[
(x,y) \approx (x',y') \ \text{if and only if}\ (x',y') \equiv (\pm x, \pm y)  \ (\text{mod}\  (2p,2q). 
\]
This time the set of equivalence classes is not a group. However, the set of equivalence classes $\Z_{2p}\times\Z_{2q}/\approx$ is a model of the billiard table. To model the billiard trajectory we use $\Z_{2pq}/\simeq$ where $n\simeq n'$ if and only if $n\equiv \pm n'\ (\text{mod}\ 2pq)$ and the billiard trajectory is just the map
\[
\alpha': \Z_{2pq}/\simeq\ \to \Z_{2p}\times\Z_{2q}/\approx, 
\]
Given by $\alpha'([n])=[(n,n)]$.

We showed in Lemma~\ref{evaft} that the astroid path moved through a grid point $(x,y)$ exactly when $x-y \equiv 0 \ (\text{mod}\ 2)$. This condition is preserved under the relation $\approx$ and is satisfied by exactly half of the points. To add a bit more here note that se can divide $\Z_{2p}\times\Z_{2q}/\approx$ into the set of equivalence classes $[(x,y)]$ with $x-y \equiv 0 \ (\text{mod}\ 2)$ and the  set of equivalence classes $[(x,y)]$ with $x-y \equiv 1 \ (\text{mod}\ 2)$. We can get from the even collectionn to the odd collection by  $[(x,y)]\mapsto [(x+1,y)]$ and from the odd on to the even one by $[(x,y)]\mapsto [(x-1,y)]$ (one checks that these respect the equivlence relation). Since the trajectoy hits all of the points with $x-y \equiv 0 \ (\text{mod}\ 2)$, it hits exactly half. This discussion establishes
\begin{lemma}
A billiard path passes through exactly half of the grid points.
\end{lemma}

What is more relevant is the fact the billiard path passes through exactly half of the \emph{interior} grid points. to see this note that the billiard path passes through exactly half of the bounary grid points. To see this cycle the boundary grid points one step in the counter-clockwise direction. (This means $(x,0)\mapsto (x+1,0)$ for $x<p$, $(p,y)\mapsto (p,y+1)$ for $y<q$, $(x,q)\mapsto (x-1,q)$ for $x>0$, and $(0,y)\mapsto (0,y-1)$ for $y>0$.) This map has an inverse given by a $1$-step clockwise cycle. The map also takes even points to odd points and odd points to even points. Thus restricting the map to the even boundary grid points gives a bijection between the  even boundary grid points and the odd boundary grid points. This means the billiard path passes through exactly half of the boundary grid points. Since it passes through exactly half of the grid points, it also passes through exactly half of the interior grid points (namely the even ones). Since there are $(p-1)(q-1)$ interior grid points, there must be $(p-1)(q-1)/2$ odd interior grid points. The is exactly the number of atomic squares. We have not given a precise definition of an atomic square. One cheap way to do so would be to define an atomic square as an odd interior grid point $(x,y)$.
Indeed, give such a point one gets a square contained within the grid having corners $(x-1,y)$, $(x,y-1)$, $(x+1,y)$, and $(x,y+1)$. As each of these are even grid points, the billiard path passes through them. One can even see that there is a segment of the  billiard path with each slope $+1$ and $-1$, so the edges of this square are contained in the billiard path. Conversely, a square with corners $(x-1,y)$, $(x,y-1)$, $(x+1,y)$, and $(x,y+1)$ in the billiard path must have each corner even so that $(x,y)$ is an odd interior grid point. This completes the proof of Theorem A. \hfill $\square$

\section*{Molecular Rectangles}

To briefly review, larger slanted rectangles composed of atomic squares are called \emph{molecular rectangles}. We begin by counting each $m\times n$ molecular rectangle within the $p\times q$ rectangular billiards table. 

\begin{definition}
\emph{Molecular Rectangles} in a $p\times q$ grid with the gcd($p,q$) = 1 are defined as having dimensions  $m\times n$ where $m$ is the dimension of the side with the negative slope and $n$ is the dimension of the side with the positive slope.

\end{definition}

We are able to count the number of $m\times n$ \emph{molecular rectangles} for each possible $m$ and $n$ and then use summations to discover a formula that counts all \emph{molecular rectangles} in the described grid. 

\begin{theorem}
The number of \emph{molecular rectangles} with dimensions $m\times n$ within the $p\times q$ grid with the gcd($p,q$) = 1 is 

\[
\text{Number of Molecular Rectangles} = \frac{1}{2}(p-m-n+1)(q-m-n+1)+\frac12\delta_{p\equiv_2q, \ m\not\equiv_2n}.
\]

where $\delta_{p\equiv_2q, \ m\not\equiv_2n}$ is equal to $1$ if both subscripts are true and $0$ otherwise.

\end{theorem}

\begin{proof}
Similar to the proof of Theorem~\ref{thmA}, we track the centers of the $m \times n$ rectangles. First consider where the center of such a rectangle could lie without regard to the billiard trajectory. Let $p_{m,n}$ represent the number of distinct horizontal placements of an $m \times n$ tilted rectangle in a $p \times q$ grid. Similarly let $q_{m,n}$ represent the number of vertical placements. The number of tilted rectangles is $p_{m,n} \times q_{m,n}$. In the proof of Theorem~\ref{thmA} we established that $p_{1,1} = p-1$ and $q_{1,1} = q-1$. Now increasing $m$ by one will slide the center of the lower left tilted rectangle to the right one-half of a grid unit, while shifting the center of the lower right rectangle to the left one-half of a grid unit. The centers of the tilted rectangles remain one grid unit apart because it is always possible to translate a tilted rectangle one grid unit provided it remains within the grid. This means increasing $m$ by one will decrease the number of horizontal tilted rectangle placements by one, i.e., $p_{m+1,n} = p_{m,n}-1$. Similarly for vertical placements, $q_{m+1,n} = q_{m,n}-1$ and increasing $n$ by one has the similar effect, $p_{m,n+1} = p_{m,n}-1$ and $q_{m,n+1} = q_{m,n}-1$. Using induction one can establish the following formulas.
\[
p_{m,n} = p-m-n+1, \quad\quad q_{m,n} = q-m-n+1.
\]

\begin{figure}[htb]
\centering
\includegraphics[width=3.6in]{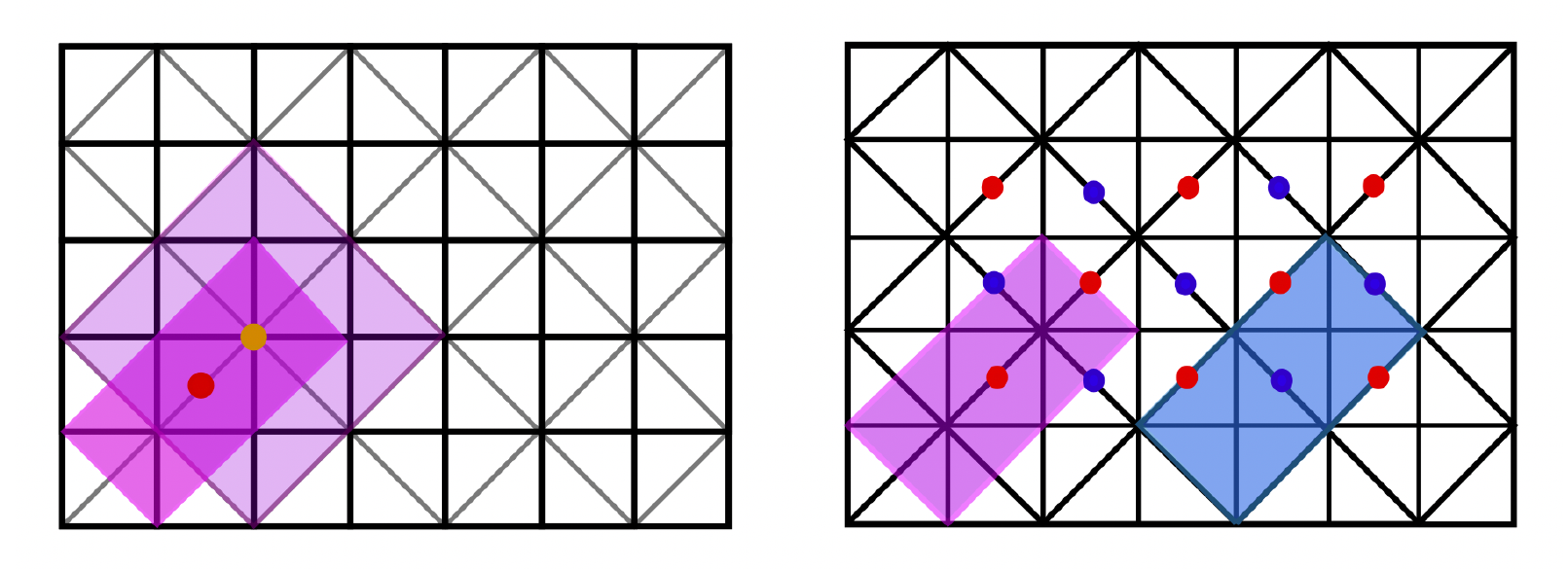}
\caption{Rectangle centers, and centers compatible with the trajectory.}
\label{MR}
\end{figure}

As in the case of atomic squares, only about half of the potential grid placements will be compatible with the billiard trajectory. The borders of the billiard trajectory appear every second grid unit moving either horizontally or vertically thus the centers of tilted rectangles may be checkerboard colored to indicate which centers are compatible with the billiard trajectory. Similar to the  proof of Theorem A, more detail could be added regarding boundary points and relating the odd interior points to the molecular rectangles. We refrain to streamline the presentation. Color these centers blue if the corresponding tilted rectangle is compatible with the billiard path, and red if it is not. All but possibly the lowest row will pair off giving the same number of red and blue dots in those upper rows. Similarly all but possibly the first dot in the lowest row will pair off. Now if both $p_{m,n}$ and $q_{m,n}$ are odd, and the lower left dot is blue then there will be $\frac12(p_{m,n}q_{m,n} +1)$ blue dots and molecular rectangles. If both $p_{m,n}$ and $q_{m,n}$ are odd, and the lower left dot is red then there will be $\frac12(p_{m,n}q_{m,n} -1)$ red dots and molecular rectangles. If one of $p_{m,n}$ or $q_{m,n}$ is even, then there will be an equal number of blue and red dots and the number of molecular rectangles will be $\frac12p_{m,n}q_{m,n}$.

If $p\not\equiv_2q$ then $p_{m,n} =p-m-n+1$ and $q_{m,n} = q-m-n+1$ will not have the same parody. Then the number of molecular rectangles will be $\frac12p_{m,n}q_{m,n}$. If $p\equiv_2q$ then $p\equiv_2q\equiv_21$ as $p$ and $q$ are relatively prime. If $p\equiv_2q\equiv_21$ and $m\equiv_2n$ then $p_{m,n} =p-m-n+1$ and $q_{m,n} = q-m-n+1$ will both be even and the number of molecular rectangles will be $\frac12p_{m,n}q_{m,n}$. However, if $p\equiv_2q\equiv_21$ and $m\not\equiv_2n$, then $p_{m,n} =p-m-n+1$ and $q_{m,n} = q-m-n+1$ will both be odd. 

The lower corner of the lower left tilted $m \times n$ rectangle is at the point $(m,0)$. In Lemma~\ref{evaft}, we saw that the billiard trajectory passes through the point $(x,y)$ exactly when $x+y\equiv_2 0$ thus the lower left tilted rectangle is compatible with the trajectory exactly when $m$ is even. This gives the result. 

\end{proof}

To establish the total number of molecular rectangles we will use the following summation formulas: 
\[
\sum_{r=1}^q1=q, \quad\quad \sum_{r=1}^qr=\frac12q(q+1), \quad \quad
 \sum_{r=1}^qr^2=\frac16q(q+1)(2q+1), \quad\quad \sum_{r=1}^qr^3=\frac14q^2(q+1)^2.
\]

\begin{proof}[Proof of Theorem B] 
The total number of molecular rectangles will be the sum of all of the $m \times n$ rectangles. If we add the number of $m \times n$ rectangles to the number of $n \times m$ rectangles, the terms $\frac12\delta_{p\equiv_2q, \ m\not\equiv_2n}$ in the $m \times n$ case and $\frac12\delta_{p\equiv_2q, \ m\not\equiv_2n}$ in the $n \times m$ case will cancel. The y coordinate of the highest point in the rectangle is $m+n$. So $m+n \le q < p$. Thus 

\[
\begin{aligned} 
& \text{Number of Molecular Rectangles} = \sum_{\substack{1\le m, 1\le n \\ m+n\le q}} \frac12(p-m-n+1)(q-m-n+1) \\ 
& \qquad = \sum_{r=2}^q \sum_{n=1}^{r-1} \frac12(p-r+1)(q-r+1) \\
& \qquad = \sum_{r=1}^q \frac12(r-1)(p-r+1)(q-r+1) \\
& \qquad = \frac12 \sum_{r=1}^q r^3-(p+q+3)r^2+(pq+2p+2q+3)r-(pq+p+q+1) \\
& \qquad = \frac12 \left[\frac14(q+1)^2q^2-\frac16q(q+1)(2q+1)(p+q+3)+\frac12q(pq+2p+2q+q)(q+1)-q(pq+p+q+1)\right] \\
& \qquad = \frac12q(q+1) \left[\frac14q(q+1)-\frac16(2q+1)(p+q+3)+\frac12(pq+2p+2q+3)-(p+1)\right] \\
& \qquad = \frac12q(q+1)(\frac14q^2+\frac14q-\frac13pq-\frac13q^2-q-\frac16p-\frac16q-\frac12+\frac12pq+p+q+\frac32-p-1) \\
& \qquad = \frac12q(q+1)(-\frac{1}{12}q^2+\frac16pq+\frac{1}{12}q-\frac16p) \\
& \qquad = \frac{1}{24}q(q+1)(q^2+2pq+q-2p) \\
& \qquad = \frac{1}{24}q(q+1)(-1)(q^2-(2p-1)q+2p) \\
& \qquad = \frac{1}{24}q(q+1)(-1)\left[(q-1)(q-2p)\right] \\
& \qquad = \frac{1}{24}q(q^2-1)(2p-q) 
\end{aligned}
\]

\end{proof}


The ``On Reflection" activity from Tanton's book \cite{Tanton} was originally designed for high school students. Even though we used a little bit of group theory in the analysis, this work should be understandable to a large audience, including high school students. One of the advantages of the ``On Reflection" activity was that it led to a number of interesting questions. Two of these questions were explored in our research project. The simple form for the expression for the number of molecular rectangles makes us wonder if there is a more direct explanation for this expression. 

There is a large body of research related to mathematical billiards. One fundamental question that has been studied since antiquity is the prediction of the orbits of the planets, and the existence of periodic orbits. Poincare\' and Birkhoff did fundamental research on this question forming the foundation of mathematical billiards. Discussion of this history and connections to advanced mathematical descriptions of mechanics may be found in Arnold's classic book \cite{arn}. A shorter expository paper about mathematical papers is the one by Katok, \cite{Kat}. Two of the many books on the subject are by Rozikov and Tabachnikov,\cite{Roz,Tab}.


\section*{Biographies}

\medskip\noindent{\bf David Auckly} has directed undergraduate research projects for nearly forty years. He enjoys sharing the beauty of mathematics with people. He is co-founder and co-director of the Navajo Nation Math Circles project. Aside from mathematics, he enjoys outside pursuits. 

\medskip\noindent{\bf Betsy Giles} is 2021 graduate from Kansas State University, with a B.S. in 
Biochemistry. She found that understanding mathematics provides a deeper 
understanding of the world around us and spent time during her undergraduate 
studying mathematical billiards with Dr. Auckly. She is currently in Israel pursing 
her M.D. with an emphasis in Global Health at Ben Gurion University.

\end{document}